\documentclass[11pt,reqno]{amsart}

\usepackage{subfigure}

\usepackage{amsmath,amscd}
 \usepackage{amsfonts}
 \usepackage{amssymb}
 \usepackage{amscd}
 \usepackage[all]{xy}
\usepackage{colordvi}
\usepackage{pdfsync}
\usepackage{url}
\usepackage{bbm}
\usepackage[colorlinks]{hyperref}  

\newcommand{\modulo}[1]{\quad (\mbox{mod }{#1})}

\newtheorem{theorem}{Theorem}[section]

\newtheorem{lemma}[theorem]{Lemma}
\newtheorem{corollary}[theorem]{Corollary}
\newtheorem{conjecture}[theorem]{Conjecture}
\newtheorem{proposition}[theorem]{Proposition}
\newtheorem{question}[theorem]{Question}
\newtheorem{problem}[theorem]{Problem}

\theoremstyle{definition}
\newtheorem{definition}[theorem]{Definition}

\newtheorem*{remark}{Remark}
\newtheorem{example}{Example}

\def \N { {\mathbb N} }
\def \Q { {\mathbb Q} }
\def \R { {\mathbb R} }
\def \C { {\mathbb C} }
\def \Z { {\mathbb Z} }

\newcommand{\<}{\langle}
\renewcommand{\>}{\rangle}
\newcommand{\old}[1]{}

\newcommand{\modulonospace}[1]{\;(\mbox{mod }{#1})}

\author[Hillar]{Christopher J. Hillar}
\address{The Mathematical Sciences Research Institute\\
         17 Gauss Way\\
         Berkeley, CA 94720-5070\\
         USA}
\email{chillar@msri.org}
\urladdr{http://www.msri.org/people/members/chillar}
\author[Levine]{Lionel Levine}
\address{Department of Mathematics  \\
         Massachusetts Institute of Technology \\
         Cambridge, MA USA}
\email{levine@math.mit.edu}
\urladdr{http://math.mit.edu/~levine/}
\author[Rhea]{Darren Rhea}
\address{University of California, Berkeley \\
         Berkeley, CA 94720 \\
         USA}
\email{darren.rhea@gmail.com}

\thanks{This research was conducted while the first author was supported by an NSA Young Investigators Grant and an NSF All-Institutes Postdoctoral Fellowship administered by the Mathematical Sciences Research Institute through its core grant DMS-0441170.  The second author was partially supported by an NSF Postdoctoral Fellowship.}

\title[Equations solvable by radicals]{Equations solvable by radicals \\ in a uniquely divisible group}

\date{March 26, 2012}



\begin{document}

\begin{abstract}
We study equations in groups $G$ with unique $m$-th roots for each positive integer $m$.
A \textit{word equation} in two letters is an expression of the form
$w(X,A) = B$, where $w$ is a finite word in the alphabet $\{X,A\}$.
We think of $A,B \in G$ as fixed coefficients, and $X \in G$ as the
unknown.  Certain word equations, such as $XAXAX=B$, have solutions in
terms of radicals: \[ X = A^{-1/2}(A^{1/2}BA^{1/2})^{1/3}A^{-1/2} \]
while others such as $X^2 A X = B$ do not.  We obtain the first known
infinite families of word equations not solvable by radicals, and
conjecture a complete classification.

To a word $w$ we associate a polynomial ${P}_w \in \Z[x,y]$ in two
\emph{commuting} variables, which factors whenever $w$ is a
composition of smaller words.  We prove that if $P_w(x^2,y^2)$ has an absolutely irreducible factor in $\Z[x,y]$, then the equation $w(X,A)=B$ is not solvable in terms of radicals.
\end{abstract}

\maketitle

\section{Introduction}

A group~$G$ is called \textit{uniquely divisible} if for every $B \in G$ and each positive integer
$m$, there exists a unique $X \in G$ such that \[X^m = B.\]  
We denote the unique such~$X$ by $B^{1/m}$, and its inverse by $B^{-1/m}$.  In the literature, such groups are also referred to as $\Q$-groups. Note that if it is not the trivial group, then $G$ must be torsion-free, hence infinite.  Examples of uniquely divisible groups include the group of positive units of a real closed field, unipotent matrix groups, noncommutative power series with unit constant term, and the group of characters of a connected Hopf algebra over a field of characteristic zero~\cite{Aguiar}. 

Inspired by trace conjectures in matrix analysis (see section \ref{sec:background}),
we study here the natural question of which equations in a uniquely divisible group have solutions in terms of radicals.  As a motivating example, consider the Riccati equation \[ XAX = B \] with $A,B \in G$ given and $X \in G$ unknown.  This equation has a unique solution~$X$ in any uniquely divisible group; moreover, its solution may be written explicitly (albeit in two distinct ways) as 
\begin{equation}\label{riccatieq}
X = A^{-1/2}(A^{1/2}BA^{1/2})^{1/2}A^{-1/2} = B^{1/2}(B^{-1/2}A^{-1}B^{-1/2})^{1/2}B^{1/2}.
\end{equation}

More generally, let $w(X,A)$ be a finite word in the two-letter alphabet $\{X,A\}$.  An expression of the form
\begin{equation} \label{wordeq} w(X,A) = B \end{equation}
is called a \textit{word equation}.  We are interested in classifying those word equations that have a solution in every uniquely divisible group.  Clearly, the more general situation in which positive rational exponents on $X$, $A$, and $B$ are allowed reduces to this one.

The main tool in our analysis is a new combinatorial object $P_w \in \Z[x,y]$, called the \textit{word polynomial}.  If $w = A^{a_0} X A^{a_1} X \cdots A^{a_{n-1}} X A^{a_{n}}$, we define
	 \[ P_w(x,y) :=  y^{a_0} + xy^{a_0+a_1} + x^2y^{a_0+a_1+a_2} + \cdots + x^{n-1} y^{a_0 + \cdots + a_{n-1}}. \]
For a prime $p$, let $(\Z/p\Z)^*$ denote the set of nonzero elements of the finite field $\mathbb F_p = \Z/p\Z$.  Recall that a group is called \emph{metabelian} if its commutator subgroup is abelian. The following is our main result.  

\begin{theorem}\label{thereduction}
There exists a uniquely divisible metabelian group $G$ with the following property: For all finite words $w$ in the alphabet $\{X,A\}$, if the equation \[ P_w(x^2,y^2)=0 \] has a solution $(x_p,y_p) \in (\Z/p\Z)^* \times (\Z/p\Z)^*$ for all but finitely many primes $p$, then there exist elements $A,B,B' \in G$ for which the word equation \[ w(X,A) = B \] has no solution $X \in G$, and the  word equation 	\[ w(X,A) = B' \] 
has at least two solutions $X \in G$.
\end{theorem}

We prove this theorem in section~\ref{sec:reductionthm}.  The group $G$ is constructed from an infinite collection of $pq$-groups whose orders are chosen using Dirichlet's theorem on primes in arithmetic progressions.  

In a uniquely divisible abelian group, every word equation has a unique solution.  Theorem~\ref{thereduction} shows that this is already not the case for uniquely divisible metabelian groups: By Example~\ref{X2AXex}, below, there exist $A,B \in G$ such that the equation $X^2 A X = B$ has no solution $X \in G$.

Despite appearances, Theorem \ref{thereduction} yields a computationally efficient sufficient condition for the equation $w(X,A)=B$ to have no solution in~$G$ (see Corollary \ref{mainUnivCor}).  To put Theorem~\ref{thereduction} in context, we next discuss a family of word equations which \emph{do} have solutions in every uniquely divisible group, along with a conjectured complete classification of such words.

We believe that the word polynomial $P_w$ may be useful in other settings.  It behaves particularly well under composition of words: if $u \circ w$ is the word obtained by substituting $w$ for each instance of $X$ in the word $u$, then $P_{u \circ w}$ has a simple factorization in terms of $P_u$ and $P_w$ (Lemma~\ref{compositiontomultiplication}).

In this paper, ``word'' will always mean a finite word over the alphabet $\{X,A\}$ (unless another alphabet is specified).
The word $w$ is called \textit{universal} if (\ref{wordeq}) has a solution $X \in G$ for every uniquely divisible group~$G$ and each two elements~$A, B \in G$; if this solution is always unique, then we say that~$w$ is \emph{uniquely universal}.   A related class of words is those for which (\ref{wordeq}) has a solution ``in terms of radicals."  This notion is defined carefully in section~\ref{sec:radWords}.  Our explorations give evidence for the surprising conjecture that all three of these classes are in fact the same, and can be characterized as follows.

\begin{definition}\label{def:decomposable}
A word $w$ in the alphabet $\{X,A\}$ is \textit{totally decomposable} if it is the image of the letter $X$ under a composition of maps of the form
\begin{itemize}
\item  $\pi_{m,k} : w \mapsto (wA^k)^mw$, for $m \geq 1$, $k \geq 0$
\item $r : w \mapsto wA$
\item $l : w \mapsto Aw$.
\end{itemize}
\end{definition}

For example, the word $w = XAX^2AXAXAX^2AX$ is totally decomposable, as witnessed by the composition $w = \pi_{1,1} \circ \pi_{1,0} \circ \pi_{1,1} (X)$.  According to the following lemma, any totally decomposable word is uniquely universal.

\begin{lemma}\label{solvablewordprop}
Let $G$ be a uniquely divisible group, and let~$w$ be a totally decomposable word.  For any $A,B \in G$, the equation $w(X,A) = B$ has a unique solution $X \in G$, and this solution can be expressed in terms of radicals. 
\end{lemma}

For the proof, see section~\ref{sec:radWords}.  One easily deduces, for example, that the word equation $XAX^2AXAXAX^2AX = B$ has a unique solution in every uniquely divisible group: \[X = A^{-1/2}\left\{ A^{1/2} \left[ A^{-1/2}(A^{1/2}BA^{1/2} )^{1/2}A^{-1/2} \right]^{1/2} A^{1/2} \right\}^{1/2} A^{-1/2}.\]

We conjecture the converse: totally decomposable words are the only universal words.

\begin{conjecture}\label{universalconj}
Let $w$ be a finite word in the alphabet $\{X,A\}$.  The following are equivalent.
\begin{enumerate}
\item $w$ is totally decomposable.
\item $w$ is uniquely universal.
\item $w$ is universal.
\item $w(X,A) = B$ has a solution in terms of radicals.  (see Definition \ref{def:radicalWord})
\end{enumerate}
\end{conjecture}

The implications $(1) \Rightarrow (2) \Rightarrow (3)$  and $(3) \Leftrightarrow (4)$ are straightforward (see section~\ref{sec:radWords}).  The remaining implication $(4) \Rightarrow (1)$ is the difficult one.  Theorem~\ref{thereduction} arose out of our attempts to prove this implication.  It reduces the noncommutative question about word equations to a commutative question about solutions to polynomial equations mod~$p$.  More concretely, we use Theorem~\ref{thereduction} to prove that several infinite families of word equations are not solvable in terms of radicals (Corollary~\ref{cor:infexamples}).  To our knowledge, these are the first such infinite families known.

Together with Theorem~\ref{thereduction}, the following would imply Conjecture~\ref{universalconj}.

\begin{conjecture}\label{notsolvableConj}
If $w$ is a word that is not totally decomposable, then the equation $P_w(x^2,y^2)=0$ has a solution $(x_p,y_p) \in (\Z/p\Z)^* \times (\Z/p\Z)^*$ for all but finitely many primes $p$.
\end{conjecture}

The questions outlined above (e.g., asking whether a particular word $w$ is universal, or uniquely universal, or solvable in terms of radicals) are examples of decidability questions in first-order theories of groups.  Determining whether a set of equations has a solution in a group is known as the \textit{Diophantine problem}.  More generally, given a set of axioms for a class of groups, one would like to provide an algorithm which decides the truth or falsehood of any given sentence in the theory.  That such an algorithm exists for free groups follows from pioneering work of Kharlampovich and Myasnikov~\cite{KharMyas06} (see also the independent work of Sela~\cite{Sela06}), but it is still open whether one exists for free uniquely divisible groups~$F^{\Q}$.  The Diophantine problem for~$F^{\Q}$ does admit such an algorithm \cite{KharMyas98} although the time complexity of the algorithm described there is likely at least doubly exponential (the proof uses the decidability of Presburger arithmetic).  In contrast, Conjecture \ref{universalconj} says that the Diophantine problem of a single equation in one variable reduces to an easily verifiable combinatorial condition (total decomposability).

\begin{example}\label{X2AXex}
Consider the word $w = X^2AX$, which is not totally decomposable; we use Theorem \ref{thereduction} to show that $w$ is not universal.  Its word polynomial is 
	 \[ P_w(x,y) = 1 + x +x^2y. \]
We need to verify that the equation $1 + x^2 +(x^2y)^2=0$ has a nonzero solution modulo~$p$, for all sufficiently large primes~$p$.  A standard pigeonhole argument shows that for all primes $p$, there exist $a,b \in \Z$ with $a\neq 0$ such that $1 + a^2 + b^2 \equiv 0 \modulonospace{p}$.  If $b = 0$ and $p \geq 5$, then $1 + (-1+4^{-1})^2 + (a+a4^{-1})^2 \equiv 0 \modulonospace{p}$ so that we may assume both $a,b$ nonzero.  Setting $x = a$ and $y = ba^{-2}$ gives us our solution.

The word polynomial for the totally decomposable word $v = XAXAX$, on the other hand, is 
	\[P_v(x,y) = 1 + xy +x^2y^2.\] 
Let $p$ be a prime greater than~$3$.  If $x,y \in (\Z/p\Z)^*$ satisfy $P_v(x^2,y^2)=0$, then setting $z=x^2y^2$ we have $z^3 = 1$ and $z \neq 1$, which forces $p \equiv 1 \modulonospace{3}$.
\qed \end{example}

Recall that a polynomial over a field $K$ is \textit{absolutely irreducible} if it remains irreducible over every algebraic extension of $K$.  The next result shows that to verify Conjecture~\ref{notsolvableConj} for a particular word~$w$, it suffices to prove that a factor of $P_{w}(x^2,y^2)$ is absolutely irreducible.

\begin{proposition}\label{WeilConjCor}
Suppose $F\in \Z[x,y]$ satisfies $F(0,0)\neq 0$, and $F$ has a factor $f \in \Z[x,y]$ which is irreducible over $\C[x,y]$. Then the equation $F(x,y)=0$ has a solution $(x_p,y_p) \in (\Z/p\Z)^* \times (\Z/p\Z)^*$ for all but finitely many primes $p$.
\end{proposition}
\begin{proof}
By a theorem of Ostrowski \cite{Ostrowski19}, the factor $f(x,y)$ is absolutely irreducible modulo all sufficiently large primes $p$ (see also \cite[Corollary 2.2.10]{Marker10} for a modern perspective from model theory).  For such a prime $p$, the homogenization $h(x,y,z) \in \Z[x,y,z]$ of $f$ is also absolutely irreducible modulo $p$, and thus defines an irreducible projective curve over $\overline{\mathbb F}_p$.  By results of Bach \cite{Bach96} (see also Aubry-Perret \cite{AubPer96}), extending a classical theorem of Weil, the curve defined by $h$ has at least $p - d^2 \sqrt{p}$ projective points all of whose coordinates lie in $\mathbb F_p$, where  $d$ is the degree of $f$.  Since $F(0,0)\neq 0$, the polynomial $f$ is not divisible by $x$ or $y$, so at most $3d$ of these points lie on the $x$-axis, $y$-axis, or the line at infinity.  It follows that for all sufficiently large primes $p$, there exists a point $(x_p,y_p) \in (\Z/p\Z)^* \times (\Z/p\Z)^*$ with $f(x_p,y_p)=0$.
\end{proof}

\begin{corollary}\label{mainUnivCor}
If $w$ is a word in the alphabet $\{X,A\}$ beginning with $X$, and if $P_{w}(x^2,y^2)$ has a factor $f \in \Z[x,y]$ such that $f$ is irreducible in $\C[x,y]$, then~$w$ is not universal.
\end{corollary}

\begin{example}\label{X2AXex2}
We show the usefulness of Corollary~\ref{mainUnivCor} by revisiting Example \ref{X2AXex}.  The word $w=X^2AX$ has \[ P_w(x^2,y^2) = 1 + x^2 + x^4 y^2, \] which is irreducible over $\C$ (since $1+x^2$ is not a square in $\C(x)$).  It follows that $X^2AX$ is not universal.
In contrast, the totally decomposable word $v=XAXAX$ has \[ P_v(x^2,y^2) = 1 + x^2 y^2 + x^4 y^4 = (1 + xy + x^2 y^2)(1-xy+x^2y^2). \] 
Each factor on the right side is irreducible over $\Z$ but factors over $\C$.
\qed \end{example}

In Section \ref{sec:notsolvable}, we use Corollary \ref{mainUnivCor} to verify Conjecture \ref{notsolvableConj} for the following infinite families of words. 


\begin{corollary}\label{cor:infexamples}
The following families of words do not have their equations solvable in terms of radicals:  
\begin{align*}
&X^nAX^m, && m,n \geq 1, \ m\neq n; \\
&XA^{m+2n}XA^{m+n}XA^mX, && m \geq 0, n \geq 1; \\
&XAX^nAX, && n \geq 3; \\
&X^2(AX)^nX, && n \geq 2.
\end{align*}

\old{
\begin{enumerate}
\item $X^nAX^m, \quad m,n \geq 1, \ m\neq n$;
\item $XA^{m+2n}XA^{m+n}XA^mX, \quad m \geq 0, n \geq 1$;
\item $XAX^nAX, \quad n \geq 3$;
\item $X^2(AX)^nX, \quad n \geq 2$.
\end{enumerate}
}
\end{corollary}

Using Corollary~\ref{mainUnivCor} and the symbolic computation software Maple, we have also verified Conjecture~\ref{universalconj} for all words of length at most~$10$.  The most difficult part of the computation is to check whether a given bivariate polynomial over~$\Z$ is irreducible over~$\C$.  This can be done in polynomial-time using the algorithm of Gao \cite{Gao03} (and is implemented in Maple).


We do not know if the condition in Corollary \ref{mainUnivCor} is sufficient to prove Conjecture \ref{universalconj}.

\begin{question}\label{notsolvableConj2}
If the word $w$ is not totally decomposable, must $P_{w}(x^2,y^2)$ have a factor in $\Z[x,y]$ which is irreducible over $\C[x,y]$?
\end{question}

The remainder of the paper is organized as follows.  In section~\ref{sec:background}, we give additional motivation arising from the BMV trace conjecture in quantum statistical mechanics.  In section~\ref{sec:radWords}, we review the basic properties of uniquely divisible groups and construct a free uniquely divisible group on two free generators.  This construction allows us to define the notion of solvability in terms of radicals, but it is not needed for the proof of Theorem~\ref{thereduction}.  Section~\ref{sec:examples} describes some important examples of uniquely divisible groups.  Most of the standard examples have the property that \emph{every} word equation with nonnegative exponents has a unique solution; the need to construct more exotic groups is part of what makes Conjecture~\ref{universalconj} so difficult.  Section~\ref{sec:wordpoly} discusses properties of the word polynomial $P_w$, and section~\ref{sec:reductionthm} is devoted to the proof of Theorem~\ref{thereduction}.  These two sections form the heart of the paper.  Finally, Section \ref{sec:notsolvable} contains the proof of Corollary~\ref{cor:infexamples}.

\section{Background and Motivation}\label{sec:background}

In this section we explain briefly the connections to matrix theory which inspired our work.  None of this material will be needed for the sequel, so the reader may safely skip to Section~\ref{sec:radWords} if desired.  The Lieb-Seiringer formulation \cite{LiebSeir04} of the long-standing Bessis-Moussa-Villani (BMV) trace conjecture \cite{BMV75, KJC80, Moussa00, otherBMVtry2, HillBMV, Haegele, LS, KS, KS2, Dykema} in statistical physics says that the trace of $S_{m,k}(A,B)$, the sum of all words of length $m$ in $A$ and $B$ with $k$ $B$s, is nonnegative for all $n \times n$ positive semidefinite matrices $A$ and $B$.  In the case of $2 \times 2$ matrices, more is true: every word in two positive semidefinite letters has nonnegative trace (in fact, nonnegative eigenvalues). 
It was unknown whether such a fact held in general until \cite{HillJohn02} appeared where it was found (with the help of Shaun Fallat) that the word $w = BABAAB$ has negative trace 
with the positive definite matrices
\[A_1=\left[\begin{array}{@{}ccc@{}}
1&20&210\\
20&402&4240\\
210&4240&44903
\end{array}\right]\quad \mbox{and}\quad B_1=\left[\begin{array}{@{}ccc@{}}
36501&-3820&190\\
-3820&401&-20\\
190&-20&1
\end{array}\right].\]  

Finding such examples is surprisingly difficult, and randomly generating millions of matrices (from the Wishart distribution) fails to produce them.  
Nonetheless, it is believed that most words can have negative trace, and it was conjectured \cite{HillJohn02} that 
if a word has positive trace for every pair of real positive definite $A$ and $B$, then it is a palindrome or a product of two palindromes (the converse is well-known).  If we replace the words ``positive trace" in the previous sentence with ``positive eigenvalues," we obtain a weaker conjecture which was also studied in \cite{HillJohn02}.  Further evidence for this conjecture can be found in \cite{HillJohn03}, where it was proved that a generic word has positive definite complex Hermitian matrices $A$ and $B$ giving it a nonpositive eigenvalue.  

Positive definite matrices which give a word a negative trace are also potential counterexamples to the BMV conjecture, and it is useful to be able to generate these matrices (see \cite[\textsection 4.1]{Hansen06} and \cite[\textsection 11]{ArmHill07} for two such examples).  As remarked above, this is difficult since random sampling does not seem to work.  This discussion explains some of the subtlety of the BMV conjecture: most words occurring in $S_{m,k} (A,B)$ likely can be made to have negative trace; however, a particular word has a small proportion of matrices which witness this. 

Although the set of $n \times n$ positive definite matrices is not a group for $n > 1$, every positive definite matrix has a unique positive definite $m$-th root for any $m$.  More remarkably, it turns out \cite{HillJohn04, ArmHill07} that every word equation $w(X,A) = B$ with $w$ palindromic (and containing at least one $X$) has a positive definite solution~$X$ for each pair of positive definite $A$ and $B$ (although this solution can be non-unique \cite{ArmHill07}).
Using $A_1$ and $B_1$, it follows that any word of the form  $wAwAAw$ with $w = w(B,A)$ palindromic (and containing at least one $B$) can have negative trace.  This gives an infinite family verifying the conjecture of \cite{HillJohn02}, and moreover, provides an infinite number of potential counterexamples to the BMV conjecture.  The existence proof in \cite{HillJohn04} uses fixed point methods, although for special cases (e.g. when $w$ contains four or less $X$s), one may express solutions $X$ explicitly (and computationally efficiently) in terms of $A, B$ and fractional powers \cite[\textsection 5]{ArmHill07}.  Computing solutions without using these formal representations ``in terms of radicals" is difficult \cite[Remark 11.3]{ArmHill07}, and it is believed that most equations do not have solutions expressable in this manner.   For instance, there is no known expression for the solution to $XAX^3AX = B$ although there is always a unique positive definite solution \cite{LawLim06}.  



\section{Radical words and the free uniquely divisible group}\label{sec:radWords}

In this section we review some basic properties of uniquely divisible groups, and construct the free uniquely divisible group on two generators.  This construction allows us to define precisely the notion of ``solvability in terms of radicals'' (but we emphasize that the proof of Theorem~\ref{thereduction} does not rely on this construction).  The following lemma shows that rational powers of group elements are well-defined and behave as expected. 

\begin{lemma}\label{lem:rationalPowLemma}
Let $G$ be a uniquely divisible group and $a \in G$.  Define $a^{n/m} := (a^{n})^{1/m}$ for $n \in \Z$ and $0 \neq m \in \N$, and define $a^0 := 1$.  Then if $p,q \in \Q$, we have $(a^{p})^{q} = (a^{q})^{p} = a^{pq}$ and $a^{p}a^{q} = a^{q}a^{p} = a^{p + q}$ . 
\end{lemma}
\begin{proof}
The first claim follows from the observation that $(a^{n})^{1/m} = ((a^{1/m})^{mn})^{1/m} = (a^{1/m})^{n}$.  To prove the second claim, it suffices to check that $a^{1/m}$ and $a^{1/\ell}$ commute for positive integers $m$ and $\ell$.  Let $b = a a^{1/m} a^{-1}$.  Then $b^m = a$, so $a$ commutes with $a^{1/m}$.  Now letting $c = a^{1/m}a^{1/\ell}a^{-1/m}$, we have $c^\ell = a^{1/m} a a^{-1/m} = a$, so $a^{1/\ell}$ commutes with $a^{1/m}$.
\end{proof}

A detailed study of uniquely divisible groups can be found in the thesis of Baumslag~\cite{Baum60} where they are called \textit{divisible R-groups}.\footnote{In fact, Baumslag studied a generalization of uniquely divisible groups to those having unique $p$-th roots for all $p\in \omega$, where $\omega$ is a subset of the primes.  We restrict ourselves to the case when $\omega$ is the set of all prime numbers.}  See also~\cite{Ledlie71} for a study of the metabelian case.  As remarked in~\cite{Baum60}, one of the difficulties is that there is no clear normal form for uniquely divisible group elements (for example, see (\ref{riccatieq}) from the introduction).  There is, however, the notion of a free uniquely divisible group which comes out of Birkhoff's theory of ``varieties of algebras'' \cite{Birkhoff35}.  Since the construction is simple, we briefly outline the main ideas here.  Our perspective is model-theoretic (see \cite{Marker10} for background) although we will use only basic notions from that subject.

Let $T$ be the first-order theory of uniquely divisible groups.  The underlying language and axioms of this theory are those of groups, with an additional (countably infinite) set of axioms expressing that every element has a unique $m$-th root for each positive integer~$m$.  Consider the smallest set~$S$ of finite, formal expressions containing letters $\{A,B\}$, exponents of the form $\phantom{}^{n/m}$ ($n \in \Z$, $0\neq m \in \N$), and balanced parentheses that is closed under taking concatenations and powers (and contains the empty expression).  For example, $S$ contains the two rightmost expressions in (\ref{riccatieq}). 

If $G$ is a uniquely divisible group and $a,b \in G$, then an expression $e = e(A,B) \in S$ defines unambiguously (by Lemma \ref{lem:rationalPowLemma}) an element $e(a,b) \in G$ by replacing letters $\{A,B\}$ with corresponding group elements $\{a,b\}$ and then evaluating the result in $G$.   When two expressions $e,f \in S$ evaluate to the same group element for each pair $a,b \in G$ in every uniquely divisible group $G$, we write $e \sim f$.  For instance, the two rightmost expressions in (\ref{riccatieq}) are equivalent in this way.  Although we will not need it here, G\"odel's completeness theorem 
(along with soundness) implies that $e \sim f$ if and only if there is a (finite) formal proof from the axioms of $T$ that they are equal.

Note that $\sim$ is an equivalence relation on $S$, and we write $[e]$ for the equivalence class containing $e \in S$.  

\begin{definition}\label{defn:freeUDG}
The set $\mathcal{F} := \{ [e] : e \in S\}$ with multiplication $[e] \cdot [f] = [ef]$ is called the \textit{free uniquely divisible group} on letters $L = \{A,B\}$.
\end{definition}

The definition extends in the obvious way to define the free uniquely divisible group on any set $L$, but (except for a remark at the very end of the paper) we shall only use the case of two generators.
 
The main facts about $\mathcal{F}$ that we will need are summarized in the following lemma.  We remark that any homomorphism $\psi: F \to G$ between uniquely divisible groups is easily seen to satisfy $\psi(a^q) = \psi(a)^q$ for all $a\in F$ and $q \in \Q$. 

\begin{lemma}\label{lem:mainFreeUDG}
$\mathcal{F}$ is a uniquely divisible group.  Moreover, $\mathcal{F}$ satisfies a universal property with respect to the map $\theta: L \to \mathcal{F}$ sending $A \mapsto [A]$ and $B \mapsto [B]$:  Given any uniquely divisible group $G$ and any map $\phi: L \to G$, there exists a unique homomorphism (of uniquely divisible groups) $\psi: \mathcal{F} \to G$ such that $\psi \circ \theta = \phi$.
\end{lemma}
\begin{proof}
That the multiplication in Definition \ref{defn:freeUDG} is well-defined and that $\mathcal{F}$ forms a group follow directly from the definitions.  Moreover, every element $[e] \in \mathcal{F}$ has an $m$-th root, namely $[(e)^{1/m}]$.  If $[f]$ is another $m$-th root of $[e]$, then for every uniquely divisible group~$G$ and all $a,b\in G$, we have $f(a,b)^m = e(a,b)$ and thus $f(a,b) = (e(a,b))^{1/m}$.  It follows that $[f] = [(e)^{1/m}]$, and thus~$\mathcal{F}$ is uniquely divisible.

Finally, one checks from the definitions that \[\psi([e(A,B)]) := e(\phi(A),\phi(B)), \ \ [e] \in \mathcal{F}\] is a well-defined homomorphism that satisfies the universal property described.
\end{proof}

This discussion allows us to formally define the concept of solution in terms of radicals mentioned in the introduction (specifically, in the statement of Conjecture \ref{universalconj}).

\begin{definition}\label{def:radicalWord}
A word $w$ is called \textit{radical} (and has equation $w(X,A)=B$ \textit{solvable in terms of radicals}) if the equation $w(X,[A]) = [B]$ has a solution $X \in \mathcal{F}$.
\end{definition}

We now connect this definition with the idea of word equations having solutions in radicals.
Let $G$ be a uniquely divisible group.  A subgroup $R \subseteq G$ is called \emph{radical} if $x^m \in R$ implies $x \in R$ for all $x\in G$ and all positive integers $m$.   Given any subset~$H$ of~$G$, recursively define sets~$R_n$ for $n\geq 0$ by setting~$R_0 = H$ and
\[ R_{n+1} = \{ (xy)^q : \, x,y \in R_{n}, \ q \in \Q\}.\]
We call the union $\mathcal{R}(H) := \bigcup_{n \in \N} R_n$ the \textit{radical subgroup of $G$ generated by $H$}. 
One easily checks that $\mathcal{R}(H)$ is a radical subgroup of $G$ and that it is the intersection of all radical subgroups containing $H$. 

For a uniquely divisible group $G$ and $a,b \in G$, the subgroup $\mathcal{R}(\{a,b\})$ can be thought of as the radical expressions generated by $a$ and $b$.  Given a specific instance of the word equation $w(x,a) = b$, any solution $x \in \mathcal{R}(\{a,b\})$ can be viewed as one ``in terms of radicals."  Of course, whether a particular word equation in a group has a solution in terms of radicals in this sense depends on the group (and the elements $a,b \in G$).  However, as the next lemma shows, a radical word always has such a solution.  Note that this verifies the implication $(4) \Rightarrow (3)$ in Conjecture~\ref{universalconj}.

\begin{lemma}
Let $w$ be a radical word, and let $G$ be a uniquely divisible group.  For any $a,b \in G$, the equation $w(x,a) = b$ has a solution $x$ which lies in the radical subgroup $\mathcal{R}(\{a,b\}) \subseteq G$. 
\end{lemma}
\begin{proof}
Let $X \in \mathcal{F}$ be a solution to $w(X,[A]) = [B]$.  By Lemma~\ref{lem:mainFreeUDG}, there exists a homomorphism $\psi: \mathcal{F} \to G$ with $\psi([A]) = a$ and $\psi([B]) = b$.  Let $x = \psi(X) \in G$ and notice that
\begin{equation*}
\begin{split}
b = \psi([B])  =  \ & \psi(w(X,[A])) \\
 = \ &  w(\psi(X),\psi([A]))  \\
= \ &  w(x,a).  \\  
\end{split}
\end{equation*}
Moreover, by the construction in Lemma \ref{lem:mainFreeUDG}, it is straightforward to verify that \mbox{$\psi(\mathcal{F}) = \mathcal{R}(\{a,b\})$}.  Thus, we also have $x \in \mathcal{R}(\{a,b\})$.
\end{proof}

We close this section with a proof of Lemma~\ref{solvablewordprop}, which shows that $(1) \Rightarrow (2)$ in Conjecture~\ref{universalconj}.  As the implications $(2) \Rightarrow (3) \Rightarrow (4)$ are trivial, the sole unproved implication is $(4) \Rightarrow (1)$. 

If~$u$ and~$v$ are words in the alphabet $\{X,A\}$, we define the composition $u \circ v$ as the word obtained by replacing each occurrence of the letter~$X$ in~$u$ by the word~$v$.

\begin{proof}[Proof of Lemma~\ref{solvablewordprop}]
We induct on the number of compositions involved in the word $w$, the base case $w = X$ being trivial.
If $w$ is totally decomposable and $w \neq X$, then $w = u \circ v$, where $u$ is totally decomposable and $v=\phi(X)$.  Here $\phi$ is one of the elementary morphisms $l,r,\pi_{m,k}$ in Definition~\ref{def:decomposable}.  By the inductive hypothesis, the equation $u(y,a)=b$ has a unique solution $y\in G$, and this solution can be expressed in terms of radicals, i.e., $y \in \mathcal{R}(\{a,b\})$.  Hence an element $x\in G$ solves the word equation $w(x,a)=b$ if and only if $\phi(x)=y$.   It therefore suffices to show that for any $y\in G$ there is a unique $x \in G$ such that $\phi(x)=y$.  The cases $\phi=\ell,r$ are immediate.
For the case $\phi=\pi_{m,k}$, let $z = a^{k/2} x a^{k/2}$, so that 
	\[ \phi(x) = a^{-k/2} z^{m+1} a^{-k/2}. \]
Then $x$ solves $\phi(x)=y$ if and only if $z = (a^{k/2} y a^{k/2})^{1/(m+1)}$; that is, if and only if
	\[ x = a^{-k/2} (a^{k/2} y a^{k/2})^{1/(m+1)} a^{-k/2}. \]
Moreover, $y \in \mathcal{R}(\{a,b\})$ implies $x \in \mathcal{R}(\{a,b\})$, which completes the inductive step.
\end{proof}

\section{Examples of Uniquely Divisible Groups}\label{sec:examples}

In addition to the free uniquely divisible group encountered in the previous section, there are many interesting examples of uniquely divisible groups.  We discuss several of them here, although this list is far from exhaustive.

Recall that a \textit{real closed field} is an ordered field~$K$ whose positive elements are squares and such that any polynomial of odd degree with coeffients in $K$ has a zero in $K$.  It follows from the definition that each positive element of a real closed field has a positive $m$-th root for every positive integer~$m$.  Moreover, since the field is ordered, this positive root is unique.  The group of positive elements of a real closed field is therefore uniquely divisible.  

The rest of our examples are noncommutative.  The free group $F_2$ on the alphabet $\{A,B\}$ may be embedded via the Magnus homomorphism~$\phi_M$ ~\cite{magnus} into the algebra $\Q \<\<a,b\>\>$ of noncommutative power series via $A \mapsto 1 + a$ and $B \mapsto 1 + b$.  The image of this map is a subgroup of the group $D$ of noncommutative power series with constant term~$1$.  Using the binomial series, it can be shown that~$D$ is uniquely divisible (see also Proposition \ref{prop:uniquepowerseries} below for another proof).  In particular, $F_2$ is a subgroup of a uniquely divisible group.

Our next result shows that every word equation with coefficients in $D$ has a unique solution in that set.  However, this solution might not be in the radical subgroup $\mathcal{R}(\phi_M(F_2))$ generated by $\phi_M(F_2)$.

\begin{proposition}\label{prop:uniquepowerseries}
For any $A_1,\ldots,A_m,B \in D$, the equation $\prod\nolimits_{i = 1}^m {(A_i X)} = B$ 
has a unique solution $X \in D$.
\end{proposition}

\begin{proof}
Write 
	\[ A_i = 1+ \sum_{w \in \{a,b\}^*} \alpha^i_w w(a,b), \]
	\[ B = 1+ \sum_{w \in \{a,b\}^*} \beta_w w(a,b), \]
	\[ X = 1+ \sum_{w \in \{a,b\}^*} x_w w(a,b),\]
where the sums are over all nonempty finite words in the alphabet $\{a,b\}$.  Denote the length of a word $w$ by $|w|$.  Expanding the product
	\begin{equation} \label{eq:generalword} \prod_{i = 1}^m {(A_i X)} = B \end{equation}
and equating coefficients, the coefficient of $w(a,b)$ on the left hand side is equal to $m x_w$ plus a polynomial $Q_w(\mathbf{x}, \mathbf{\alpha})$ in the variables $x_u$ for $|u|<|w|$ and the variables $\alpha^i_u$ for $|u|\leq |w|$.  Hence
we obtain for each nonempty word $w$ that 
	\[ m x_w + Q_w( \mathbf{x}, \mathbf{\alpha} ) = \beta_w. \]
Assuming, inductively, that $x_u$ has been determined for all $|u| < |w|$, this polynomial system has the unique solution
	\[ x_w = \frac{1}{m} (\beta_w - Q_w( \mathbf{x}, \mathbf{\alpha} )).\]
It follows that (\ref{eq:generalword}) has a unique solution~$X$.
\end{proof}



A natural problem in this setting is to characterize those power series in $D$ which arise as radical expressions of $1 + a$ and $1+b$.

\begin{problem}
Characterize those power series $f \in D$ which are elements of $\mathcal{R}(\phi_M(F_2))$.  
\end{problem}

Let $\psi: \mathcal{F} \to D$ be the homomorphism of uniquely divisible groups with $\psi([A]) = 1+a$ and $\psi([B]) = 1+b$ given by Lemma \ref{lem:mainFreeUDG}.  Surprisingly, while the Magnus homomorphism $\phi_M: F_2 \to D$ is an embedding of groups, it is a very old open question whether $\psi$ is also an embedding.  As far as we know, Baumslag has the best result on this problem \cite{Baum68}, giving injectivity when $\psi$ is restricted to certain one-relator subgroups of $\mathcal{F}$.  

Our next example is a matrix group.
Let $K$ be a field of characteristic $0$ and let~$UT_{n}$ be the group of $n \times n$ \textit{unipotent matrices} over~$K$.  These are the upper triangular matrices with coefficients in~$K$ with~$1$'s along the diagonal.  

\begin{proposition}\label{unipotent}
For any $A_1,\ldots,A_m,B \in UT_{n}$, the equation $\prod\nolimits_{i = 1}^m {(A_i X)} = B$ 
has a unique solution $X \in UT_{n}$.  (In particular, $UT_{n}$ is a uniquely divisible group.)
\end{proposition}

\begin{proof}
Induct on~$n$.  For $n = 1$, the claim is
obvious as $UT_{1} = \{1\}$.  Next for $n\geq 1$, let $A_{i},B \in UT_{n+1}$
$(i=1,\ldots,m)$ and write each $A_{i},B,X$ in block form
\[ A_{i} = \left[ {\begin{array}{*{20}c}
   U_{i} & u_{i}  \\
   0 & 1  \\
 \end{array} } \right], \ \ B = \left[ {\begin{array}{*{20}c}
   V & v  \\
   0 & 1  \\
 \end{array} } \right], \ \ X = \left[ {\begin{array}{*{20}c}
   Y & y  \\
   0 & 1  \\
 \end{array} } \right],
\] in which $U_{i},V  \in UT_{n}$ and $u_{i},v \in K^n$; and $Y \in UT_n$ and $y\in K^n$ are unknowns.  
When the equation $\prod\nolimits_{i = 1}^m {(A_i X)} = B$ is expanded, it takes the form
\begin{equation}
\left[ {\begin{array}{*{20}c}
   {\prod\nolimits_{i = 1}^m {(U_i Y)} } & u + {\sum\nolimits_{j = 1}^m {\left[ {\left( {\prod\nolimits_{i = 1}^{j - 1} {(U_i Y) } } \right)U_j } \right]} y}  \\
   0 & 1  \\
 \end{array} } \right] = \left[ {\begin{array}{*{20}c}
   V & v  \\
   0 & 1  \\
 \end{array} } \right],
 \end{equation}
 in which $u \in K^n$ is a function of $U_{i},u_{i}$, and $Y$ only (and thus does not depend on $y$).  
 
By the inductive hypothesis, 
the equation $\prod\nolimits_{i = 1}^m {(U_i Y) } = V$ has a unique solution $Y \in UT_{n}$.  Let $M = {\sum\nolimits_{j = 1}^m { {\left( {\prod\nolimits_{i = 1}^{j - 1} {(U_i Y) } } \right)U_j } } }$.  Since~$K$ has characteristic zero, the diagonal entries of~$M$ are all equal to~$m\neq 0$, so $M$ is invertible.  The word equation $\prod\nolimits_{i = 1}^m {(A_i X)} = B$ therefore has the unique solution $X = \left[ {\begin{array}{*{20}c}
   Y & y  \\
   0 & 1  \\
 \end{array} } \right]$,
 where $y = M^{-1}(v-u)$.
\end{proof}

\begin{remark}  
The proof of Proposition~\ref{unipotent} gives a recursive algorithm for finding the unique solution~$X$.
Moreover, the proof carries over to a slightly larger class of groups.  Let $R_1,\ldots,R_n$ be radical subgroups of~$K^*$, such that for each~$i=1,\ldots,n$, the additive semigroup generated by $R_i$ does not contain~$0$.  Let~$G$ 
be the group of $n\times n$ upper-triangular matrices over $K$ whose $i$-th diagonal entry lies in $R_i$ for all~$i$.  This group is uniquely divisible, and every word equation $\prod (A_i X) = B$ with $A_i, B \in G$ has a unique solution $X\in G$.  For example, taking $K=\R$ with $R_1 = \R_+$ and $R_2=\{1\}$, we obtain the group of affine transformations $t \mapsto at + b$, where $a\in \R_+$ and $b \in \R$.  
\end{remark}

\section{The word polynomial}\label{sec:wordpoly}

Given a finite word $w$ over the alphabet $\{X,A\}$, write 
\begin{equation}\label{exponentvector}
w = A^{a_0} X A^{a_1} X \cdots A^{a_{n-1}} X A^{a_n}
\end{equation}
for nonnegative integers $a_0, \ldots, a_{n}$.  The \emph{word polynomial} of~$w$ is the polynomial in commuting variables~$x$ and~$y$ given by 
	\[ P_w(x,y) := \sum_{k=0}^{n-1} x^{k} y^{a_0 + \cdots + a_k}. \]
For example, the word $w = X^{n-1}AX$ has word polynomial 
	\[P_w(x,y) = 1+x+x^2+\cdots+x^{n-2}+x^{n-1} y. \] 
Note that if $w$ ends in $X$ (i.e., $a_n=0$) then~$w$ can be uniquely recovered from~$P_w$.

If $u$ is another word over the same alphabet, the composition $u \circ w$ is the word obtained by replacing each occurrence of the letter~$X$ in~$u$ by the word~$w$.  Although composition of words is not commutative, it can be modeled by multiplication of polynomials according to the following lemma.  

\begin{lemma} 
\label{compositiontomultiplication}
Let $u$ and $w$ be finite words in the alphabet $\{X,A\}$ ending with~$X$, and let $m,n$ be respectively the number of letters in $w$ equal to $A,X$.  Then
\[ P_{u \circ w}(x,y) = P_u(x^n y^{m},y)P_w(x,y). \]
\end{lemma}
\begin{proof}
If $w$ is given by (\ref{exponentvector}), write
	\begin{align*} u &= A^{b_0} X A^{b_1} X \cdots A^{b_{m-1}} X; \\
			       u \circ w &= A^{c_0} X A^{c_1} X \cdots A^{c_{mn-1}} X. \end{align*}
Then for $0 \leq q \leq m-1$ and $0 \leq r \leq n-1$ we have
	\[ c_{qn+r} = \begin{cases} a_r + b_q , & r=0 \\ a_r, & r>0. \end{cases} \]
Since $m=a_0 + \cdots + a_{n-1}$, we have
	\[ c_0 + \cdots + c_{qn+r} = b_0 + \cdots + b_q + qm + a_0 + \cdots + a_r, \]
hence
	\begin{align*} P_{u\circ w}(x,y) &=
	\sum_{q=0}^{m-1} \sum_{r=0}^{n-1} x^{qn+r} y^{b_0 + \cdots + b_{q} + qm + a_0 + \cdots + a_{r}} \\   
		&= \sum_{q=0}^{m-1} x^{qn} y^{qm + b_0 + \cdots + b_{q}} \sum_{r=0}^{n-1} x^r y^{a_0 + \cdots + a_{r}} \\
		&= P_u(x^n y^{m},y) P_w(x,y). \qed
		\end{align*}
\renewcommand{\qedsymbol}{}
\end{proof}

Call a polynomial $g \in \mathbb Z[x,y]$ \textit{geometric} if for some integers $k,d \geq 1$ and $l \geq 0$, we have
	\[ g(x,y) = 1 + x^k y^l + (x^ky^l)^2 + \cdots + (x^ky^l)^d.\]
We thank an anonymous referee for the following observation and its proof. We will not make use of it in this paper, but we include it here in the hopes that it will aid in future attempts at proving Conjecture~\ref{universalconj}.

\begin{lemma}
\label{refereesobservation}
The word $w$ is totally decomposable if and only if $P_w(x,y)$ is a product of geometric polynomials and $y^j$ for some $j \geq 0$.
\end{lemma}
\begin{proof}
Noting that the maps in Definition~\ref{def:decomposable} satisfy $\pi_{m,k} \circ l = l \circ \pi_{m,k+1}$ and $\pi_{m,k} \circ r = r \circ \pi_{m,k+1}$, we can move all $l$ and $r$ maps in the composition to the beginning.  
So any totally decomposable word $w$ can be written as $w= A^j (u_1 \circ \cdots \circ u_r) A^{j'}$ for some $j,j' \geq 0$ and words $u_i = \pi_{m_i,k_i}(X)$. Noting that each word polynomial $P_{u_i}$ is geometric, it follows from Lemma \ref{compositiontomultiplication} that $P_w$ is a product of geometric polynomials and $y^j$.

The key observation for the converse is the following. If $w$ and $W$ are words ending in $X$ such that
	\begin{equation} \label{wordpolyfactor} P_W (x,y) = g(x,y) P_w (x,y) \end{equation}
for a geometric polynomial $g(x,y) = 1 + x^k y^\ell + \cdots + (x^k y^\ell)^d$, then $k=n$ and $\ell \geq m$, where $m,n$ are respectively the number of letters in $w$ equal to $A,X$. 
To see this, set $y=1$ in \eqref{wordpolyfactor} to obtain
	\[ 1+x+\cdots + x^{N-1} = (1+x^k + \cdots + x^{kd} ) (1+x+\cdots + x^{n-1}). \]
Equating coefficients of $x^k$ shows that $k \geq n$.  Equating coefficients of $x^n$ then shows that $k=n$. 
Now the coefficients of $x^{n-1}$ and $x^n$ on the right side of \eqref{wordpolyfactor} are respectively $y^m$ and $y^\ell$, which shows that $\ell \geq m$.

Setting $u = (XA^{\ell-m})^d X$ we see that $g(x,y) = P_u(x^n y^m, y)$, so $P_W = P_{u \circ w}$ by Lemma~\ref{compositiontomultiplication}. Since a word is determined by its word polynomial up to multiplication on the right by a power of $A$, we conclude that $W = u \circ w$. In particular, if $w$ is totally decomposable, then so is $W$.  It follows by induction on $r$ that if $P_w = y^j g_1 \cdots g_r$ for some $j \geq 0$ and geometric polynomials $g_1,\ldots,g_r$, then $w$ is totally decomposable.
\end{proof}

\old{
It will sometimes be useful to have the word polynomial~$P_w$ defined when~$w$ does not end in~$X$; to do this, we set $P_{wA} = P_w$ for all words~$w$, where~$wA$ is the word obtained from~$w$ by appending the letter~$A$.  
}
Our next lemma shows another context in which the word polynomial~$P_w$ arises: from substituting certain affine transformation matrices for the letters of~$w$.

\begin{lemma}
\label{matrixsubstitution}
Let $x,y,z$ be commuting indeterminates, let $w(X,A)$ be a finite word, and let $m,n$ be respectively the number of letters in $w$ equal to $A,X$.   Then
\[ w \left(\left[\begin{array}{@{}cc@{}}
x&z\\
0&1\\
\end{array}\right],\,
\left[\begin{array}{@{}cc@{}}
y&0\\
0&1\\
\end{array}\right] \right) = \left[\begin{array}{@{}cc@{}}
x^ny^m & P_w(x,y)z\\
0&1\\
\end{array}\right].\]
\end{lemma}
\begin{proof}
Note that $P_{A} = 0$, $P_{X}=1$, and $P$ obeys the recurrence
	\begin{align*}
	P_{Xw} (x,y) &= x P_{w} (x,y) + 1, \\
	P_{Aw} (x,y) &= y P_{w} (x,y).
	\end{align*}
Here~$Xw$ (resp.\ $Aw$) denotes the word obtained by prepending the letter~$X$ (resp.\ $A$) to~$w$.  The result follows by induction on $m+n$.
\end{proof}

\begin{remark}
One can also define a noncommutative word polynomial $\tilde P_w$ by taking $x,y,z$ above to be noncommuting indeterminates.  It is given by
	\[ \tilde P_w(x,y) = y^{a_0} + y^{a_0} x y^{a_1} + \cdots + y^{a_0} x y^{a_1} x \cdots x y^{a_{n-1}}. \]
Another interpretation of $\tilde P_w$ (pointed out to us by Martin Kassabov) is the Fox derivative $\partial w/ \partial X$ (see \cite{fox}) where we interpret $w(X,A)$ as an element of the free group on generators $\{X,A\}$. From this perspective, Lemma~\ref{compositiontomultiplication} is a form of the chain rule.  

We will only use the commutative word polynomial $P_w$ in this paper.  We remark that this polynomial arises in a completely different context in the theory of symmetric functions, as an eigenvalue of the first-order Macdonald operator $D_n^1$; see \cite[page 317, eq.\ (3.10)]{Macdonald}.
\end{remark}

\old{
As we will see, word polynomials appear naturally in exponent representations when ``commuting" certain product expressions in $pq$-groups; see Lemma \ref{thepolynomialappears!} below.  For this purpose, it is convenient to define the word polynomial~$P_w$ for all words~$w$, not just those ending in~$X$.  For an arbitrary word~$w$, write~$w = u A^{a_{n+1}}$, where $u$ is either the empty word or a word ending with~$X$.  If~$u$ is empty, then we set~$P_w=0$, and otherwise we set~$P_w = P_u$.
}

\section{Proof of Thoerem \ref{thereduction}}\label{sec:reductionthm}

We begin with the following elementary fact.

\begin{lemma}\label{relprimeuniqueroot}
Let $G$ be a group of order $n$.  If $m$ and $n$ are relatively prime, then every element of $G$ has a unique $m$-th root.
\end{lemma}
\begin{proof}
There exist $r, s \in \Z$ such that $rm + sn = 1$.  One now checks that $g^r$ is the unique $m$-th root of $g$. 
\end{proof}

Let $G_i$ ($i = 1,2,\ldots$) be an infinite sequence of finite groups with the following property: For every positive integer $m$, there exists an $N$ such that
\begin{equation}
\label{nosmallprimedivisors}
\mbox{$m$ and $\#G_i$ are relatively prime for all $i > N$}.
\end{equation}
  By Lemma \ref{relprimeuniqueroot}, these groups have a limiting kind of unique divisibility, which suggests taking the quotient of the direct product of the $G_i$ by their direct sum.

\begin{lemma}\label{directproduct}
If $G_1, G_2,\ldots$ is a sequence of finite groups satisfying (\ref{nosmallprimedivisors}), then the group
\[
G = \prod_{i=1}^{\infty} { G_i  } \Big/ \bigoplus_{i=1}^{\infty} { G_i  }
\]
is uniquely divisible.
\end{lemma}

\begin{proof}
Given $g = (g_1,g_2, \ldots) \in \prod G_i$, write $[g]$ for its image under the quotient map $\prod G_i \to G$.  For any $g \in \prod G_i$ and any $m \geq 1$, by Lemma~\ref{relprimeuniqueroot} and (\ref{nosmallprimedivisors}), there exists~$N$ such that for all $i\geq N$ the element $g_i$ has a unique $m$-th root $h_i \in G_i$.  Set $h_i=1$ for all $i<N$.  Then $[h]^m = [g]$.  Moreover, if $h' \in \prod G$ satisfies $[h']^m=[g]$, then $(h'_i)^m = g_i$ for all but finitely many~$i$.  By (\ref{nosmallprimedivisors}) it follows that $h'_i=h_i$ for all but finitely many $i$, so $[h']=[h]$.  Thus every element $[g]\in G$ has a unique $m$-th root in $G$.
\end{proof}

This lemma allows us to construct many examples of uniquely divisible groups.  Next we describe the sequence of groups $G_i$ that we will use to prove Theorem~\ref{thereduction}.

Let $p$ be an odd prime, and let  $q=\frac{p-1}{2}$.  Since the group $(\Z/p\Z)^*$ of units mod~$p$ is cyclic of order $p-1$, we can pick an element $t\in (\Z/p\Z)^*$ whose multiplicative order mod~$p$ is $q$ (namely, $t$ can be the square of any generator).  The powers of $t$ are exactly the nonzero squares, i.e. quadratic residues, in $\Z/p\Z$.  We take $G_p$ to be the semidirect product $(\Z/q\Z) \ltimes (\Z/p\Z)$, which has the presentation
    \[ G_p = \left\langle S,T \; : \; S^t T=TS,\, T^q=1,\, S^p=1 \right\rangle. \]

The group $G_p$ can be realized concretely as the group of affine transformations of $\Z/p\Z$ of the form 
	\[ z \mapsto a z + b \] 
where $a \in (\Z/p\Z)^*$ is a quadratic residue and $b \in \Z/p\Z$ is arbitrary.  Thus we can view $G_p$ as the group of all $2\times 2$ matrices of the form
\[ \left[\begin{array}{@{}cc@{}}
t^k&b\\
0&1\\
\end{array}\right] \]
where $k\in \Z/q\Z$ and $b \in \Z/p\Z$.  The generators $S$ and $T$ correspond to the affine transformations $z \mapsto z+1$ and $z \mapsto tz$, or the matrices
 \[	S=\left[\begin{array}{@{}cc@{}}
1&1\\
0&1\\
\end{array}\right], \qquad 
T=\left[\begin{array}{@{}cc@{}}
t&0\\
0&1\\
\end{array}\right] . \]

\old{
From the identity $t^k (z + \ell) = t^k z + t^k \ell$ we have the relation
	\[  T^k S^\ell = S^{t^k \ell} T^k. \]
}

\old{
Any element $M \in G_p$ can be uniquely expressed in ``normal form'' $M = S^\gamma T^\beta $, where $0 \leq \gamma \leq p-1$ and $0 \leq \beta \leq q-1$, i.e. basically any affine transformation in the group $G_p$ is expressible as  $z\mapsto t^\beta\cdot z + \gamma$.  

}

\begin{lemma}
\label{thepolynomialappears!}
Let $\alpha,\beta \in \{0,\ldots,q-1\}$ and $\gamma \in \{0,\ldots,p-1\}$.  For any word $w = w(X,A)$, the following identity holds in the group~$G_p$:
    \[ w\big(S^\gamma T^\beta, T^\alpha \big) = S^{\gamma P_{w}(t^{\beta},t^\alpha)} T^{\alpha m + \beta n} \]
where $m$ and $n$ are respectively the number of $A$'s and $X$'s in~$w$.
\end{lemma}

\begin{proof}
Since
\[ S^\gamma T^\beta = \left[\begin{array}{@{}cc@{}}
t^\beta & \gamma \\
0&1\\
\end{array}\right], \qquad T^\alpha =\left[\begin{array}{@{}cc@{}}
t^\alpha & 0 \\
0&1\\
\end{array}\right] \]
the result follows from setting $x=t^\beta$, $y=t^\alpha$ and $z=\gamma$ in Lemma~\ref{matrixsubstitution}.
\end{proof}

As the reader will see, the restriction that $t$ is a quadratic residue is responsible for the appearence of $P_w(x^2,y^2)$ instead of $P_w(x,y)$ in the next lemma, and hence in Theorem~\ref{thereduction}.  This restriction is essential in our construction, because the group $(\Z / p \Z)^* \ltimes (\Z/p\Z)$ has even order $p(p-1)$, so these larger groups can not be used to form a sequence satisfying (\ref{nosmallprimedivisors}).

\begin{lemma}
\label{thecrux}
Let $w$ be a finite word in the alphabet $\{X,A\}$, and let $n$ be the number of letters in~$w$ equal to~$X$.  Let~$p$ be a prime such that~$q = \frac{p-1}{2}$ is relatively prime to~$n$.
If the equation \[ P_w(x^2,y^2)=0 \] has a solution $(x,y) \in (\Z/p\Z)^* \times (\Z/p\Z)^*$, then there exist $a,b \in G_p$ for which the word equation \[ w(X,a)=b \] has no solution $X \in G_p$.
\end{lemma}

\begin{proof}
Suppose $(x,y) \in (\Z/p\Z)^* \times (\Z/p\Z)^*$ solves $P_w(x^2,y^2)=0$.
Since any quadratic residue mod~$p$ is a power of $t$, we can find integers $\alpha,\delta$ such that $x^2 = t^\delta$ and $y^2=t^\alpha$.  Let 
	\[ a = T^\alpha, \quad
	   b = S T^{\alpha m + \delta n} \]
where $m$ is the number of letters in~$w$ equal to $A$.
By Lemma~\ref{thepolynomialappears!}, an element $X = S^\gamma T^\beta  \in G_p$ solves the word equation $w(X,a)=b$ if and only if
	\begin{equation} \label{equatemypowers} S^{\gamma P_{w}(t^{\beta},t^\alpha)} T^{\alpha m + \beta n} = b = S T^{\alpha m + \delta n}. \end{equation}
Equating powers of $T$, we obtain $\beta n \equiv \delta n \modulonospace{q}$.
Since $n$ and $q$ are relatively prime, it follows that $\beta \equiv \delta \modulonospace{q}$, and hence $t^{\beta} \equiv t^\delta \modulonospace{p}$. Now equating powers of~$S$ in (\ref{equatemypowers}) yields
	\[ 1 \equiv \gamma  P_w(t^{\beta},t^\alpha) \equiv \gamma  P_w(t^\delta,t^\alpha) \equiv \gamma  P_w(x^2,y^2) \equiv 0 \modulo{p}, \]
so there is no solution~X to $w(X,a)=b$ in $G_p$.
\end{proof}

We may now prove our main result.

\begin{proof}[Proof of Theorem~\ref{thereduction}]
Let $\pi_0 = 2, \pi_1 = 3, \ldots$ be the primes in increasing order.  By the Chinese remainder theorem, for each~$i \geq 1$ there is an integer~$k_i$ satisfying
	\begin{align*} k_i &\equiv 3 \modulo{4} \\
	 k_i &\equiv 2 \modulo{\pi_j}, \qquad j=1,\ldots,i. \end{align*}
By Dirichlet's theorem on primes in arithmetic progression, for each~$i$ there exists a prime~$p_i$ satisfying
	\[ p_i \equiv k_i \modulo{4\pi_1 \ldots \pi_i}. \]
By construction, $\frac{p_i-1}{2}$ is not divisible by any of $2,\pi_1,\ldots,\pi_i$.  Since 
	\[ \# G_{p_i} = \frac{p_i (p_i -1)}{2}, \]
the sequence of groups $G_{p_1}, G_{p_2},\ldots$ satisfies condition (\ref{nosmallprimedivisors}), so by Lemma~\ref{directproduct} the quotient group
	\[ G = \prod_{i \geq 1} G_{p_i} \Big/ \bigoplus_{i \geq 1} G_{p_i} \]
is uniquely divisible.  Moreover, $G$ is metabelian since each factor $G_{p_i}$ is metabelian.

Now let~$w$ be a word in the alphabet~$\{X,A\}$, and let~$n$ be the number of letters in~$w$ equal to~$X$.  Let $\pi_{i_0}$ be the largest prime divisor of~$n$.
For $i > i_0$ we have $\frac{p_i-1}{2}$ relatively prime to~$n$.  By hypothesis, we can choose $i_1 \geq i_0$ sufficiently large so that the equation $P(x^2,y^2) = 0$ has a solution $(x_i,y_i) \in (\Z/p_i\Z)^* \times (\Z/p_i\Z)^*$ for all $i > i_1$.  By Lemma~\ref{thecrux}, for each $i > i_1$, there exist $a_i,b_i \in G_{p_i}$ for which the word equation
	\[ w(X,a_i)=b_i \]
has no solution $X \in G_{p_i}$.
%
Let $A,B \in G$ be the images of the sequences $(a_i)_{i > i_1}$ and $(b_i)_{i>i_1}$ under the quotient map
	\[ \prod_{i > i_1} G_{p_i} \longrightarrow G. \]
The equation $w(X,A)=B$ has no solution $X \in G$.

Finally, we prove that there exists $B' \in G$ such that $w(X,A) = B'$ has at least two solutions $X \in G$.  Using the result just proved, for all $i > i_1$, the map $G_{p_i} \to G_{p_i}$ sending $g \mapsto w(g,a_i)$ is not surjective, hence not injective.  Let $g_i, \tilde{g}_i$ be distinct elements of $G_{p_i}$ such that $w(g,a_i)=w(\tilde{g}_i,a_i)$, and let $B'$ be the image in $G$ of the sequence $(w(g_i,a_i))_{i>i_1}$.  Then the images in $G$ of the sequences $(g_i)_{i>i_1}$ and $(\tilde{g}_i)_{i>i_1}$ are distinct solutions to $w(X,A)=B'$.
\end{proof}

\section{Word equations not solvable by radicals}\label{sec:notsolvable}

In this section, we use Corollary \ref{mainUnivCor} to find several infinite families of words that are not universal, and consequently not solvable in terms of radicals.

\begin{lemma}\label{XnAXthm}
Let $m$ and $n$ be distinct positive integers, and let $w = X^mAX^n$.  Then $P_{w}(x^2,y^2)$ is irreducible in $\C[x,y]$.
\end{lemma}
\begin{proof}
We view the word polynomial
	\[P_w(x^2,y^2) = \frac{x^{2m}-1}{x^2-1} + y^2 x^{2m} \frac{x^{2n}-1}{x^2-1} \]
as a polynomial in~$y$ with coefficients in $\C(x)$.  

If $m<n$, then there exists $\zeta \in \C$ such that $\zeta^{2m} \neq 1$ and $\zeta^{2n} = 1$, in which case $\zeta$ is a simple pole of $f(x) = x^{-2m} (x^{2m}-1) / (x^{2n}-1)$; likewise, if $m>n$, then $f$ has a simple root.  Thus $f$ is not a square in $\C(x)$, which implies that $P_w(x^2,y^2)$ is irreducible in $\C(x)[y]$.
\end{proof}

By Corollary~\ref{mainUnivCor}, it follows that for positive integers $m\neq n$, the word equation 
	\[ X^mAX^n = B \]
has no solution in terms of radicals.  Our next result shows that for $m \geq 0$ and $n\geq 1$, the word equation 
	\[ XA^{m+2n}XA^{m+n}XA^mX = B \]
has no solution in terms of radicals.

\begin{lemma}\label{XAr2sXArsXArXthm}
Let $m \geq 0$ and $n \geq 1$ be integers, and let $w=XA^{m+2n}XA^{m+n}XA^mX$.  Then $P_{w}(x^2,y^2)$ has a factor in $\Z[x,y]$ which is irreducible over $\C[x,y]$.
\end{lemma}

\begin{proof}
The corresponding word polynomial factors as:
\[ P_{w}(x,y) = (1 + xy^{m+n})[1 + x(y^{m+2n} - y^{m+n}) + x^2 y^{2m+2n}].\]
The second factor, upon substituting $x^2$ for $x$ and $y^2$ for $y$, is of the form $h(x,y) = 1 + x^2 f(y) + x^4 g(y).$
We claim that~$h$ is irreducible over $\C[x,y]$.   Let $K = \overline{\C(y)}$ be the algebraic closure of the field of rational functions in $y$.
The polynomial $h$ factors over $K[x]$ as a product of the four linear factors
\[ x \pm \sqrt{\frac{-f(y) \pm \sqrt{D(y)}}{2}} \]
where $D(y) = f(y)^2 - 4g(y)$.  If $h$ had another factorization over $\C[x,y] = \C[y][x]$, then it would have to further factor into this one (up to units) by unique factorization in $K[x]$. In particular, this would imply that a proper subset of the above factors combine to give an element of $\C[x,y]$.  It is easy to check that this can only happen if $D(y)$ is a perfect square in $\C[y]$.
But in fact, \[D(y)  = y^{4m+4n} (y^{2n}-3)(y^{2n} + 1)\] is not a square (all of its nonzero roots are simple).
\end{proof}

The next example shows that the word equation
	\[ XAX^nAX = B \]
has no solution in terms of radicals if $n\geq 3$.

\begin{lemma}\label{XAXbAXthm}
Let $n \geq 3$ be an integer, and let $w = XAX^nAX$.  Then $P_{w}(x^2,y^2)$ is irreducible in $\C[x,y]$.
\end{lemma}

\begin{proof}
We have
\[ P_{w}(x^2,y^2) = 1 + (x^2 + x^4 + \cdots + x^{2n})y^2+ x^{2n+2}y^4.\]
As in the proof of Lemma~\ref{XAr2sXArsXArXthm}, this polynomial factors over
$\C[x,y]$ only if \[D(x) = (1 + x^2 + \cdots + x^{2n-2})^2 - 4 x^{2n-2}\] 
is a square in $\C[x]$.  Any factor $f$ of $D$ that is irreducible over $\Z[x]$ has distinct roots in $\C$.  Moreover, the zero sets in~$\C$ of any two factors of~$D$ that are irreducible over $\Z[x]$ are either equal or disjoint.  It follows that if~$D$ is a square in~$\C[x]$, then each irreducible factor over~$\Z[x]$ divides~$D$ with even multiplicity, hence~$D$ is a square in $\Z[x]$.  But this is impossible, as~$D(1)=n^2 - 4$ is not a square in~$\Z$ for~$n \geq 3$.
\end{proof}

To further extend these families of words not solvable by radicals, note that under the hypotheses of Theorem~\ref{thereduction}, we can actually derive a slightly stronger conclusion.

\begin{corollary}
Let $u,w$ be finite words in the alphabet $\{X,A\}$.  If $P_w(x^2,y^2)=0$ has a solution $(x_p,y_p) \in (\Z/p\Z)^* \times (\Z/p\Z)^*$ for all but finitely many primes $p$, then the word equations 
	\[ u\circ w(X,A)=B \ \  \text{ and } \ \ w\circ u(X,A)=B \] 
have no solution in terms of radicals.
\end{corollary}

\begin{proof}
By Lemma~\ref{compositiontomultiplication}, the word polynomial $P_w$ divides $P_{u\circ w}$, so the equation $P_{u \circ w}(x^2,y^2)=0$ has a nonzero solution modulo all sufficiently large primes.  By Theorem~\ref{thereduction}, there exist $A,B \in G$ for which the equation $u\circ w(X,A)=B$ has no solution $X \in G$. 

The statement for $w\circ u$ is immediate from Theorem~\ref{thereduction}: since there exist $A,B \in G$ for which the equation $w(X,A)=B$ has no solution $X \in G$, the equation $w\circ u(X,A) = w(u(X,A),A) = B$ also has no solution $X \in G$.
\end{proof}

A simple substitution also proves the following.

\begin{corollary}
%
%
Let $n \geq 2$ be an integer.  The word equation 
	\[ X^2(AX)^nX = B \]
is not solvable by radicals.
\end{corollary}

\begin{proof}
The solutions $X$ to $X^2(AX)^{n}X = B$ are in direct correspondence with solutions $Y$ to $YCY^{n+1}CY = D$ via the relationships $Y = A^{1/2}XA^{1/2}$, $C = A^{-1}$, $D = A^{1/2}BA^{1/2}$.  In particular, if there are no solutions to the latter equation in a uniquely divisible group $G$ for some $C,D \in G$, then there will be none to the former with $A = C^{-1}$ and $B = A^{-1/2}DA^{-1/2}$.
\end{proof}

As a final remark, one might ask about more general classes of word equations in uniquely divisible groups, such as equations of the form
	\begin{equation}  \label{moregeneral} X^{n_0} A_1 X^{n_1} A_2 \cdots X^{n_{k-1}} A_k X^{n_k} = B \end{equation}
for positive integers $k$ and $n_0, \ldots, n_k$.  However, these equations almost never have solutions in terms of radicals.

\begin{corollary}
The word equation (\ref{moregeneral}) has a solution~$X$ in terms of radicals if and only if $k=1$ and $n_0=n_1$.
\end{corollary}

\begin{proof}
Suppose first that $k>1$.  By setting $A_3 = \ldots = A_k = 1$, any solution in terms of radicals to (\ref{moregeneral}) 
yields a solution to
	\[ X^{n_0} A_1 X^{n_1} A_2 X^{m_2} = B \]
in terms of radicals, where $m_2 = n_2 + \cdots + n_k$.  If $n_0 + n_1 \neq m_2$, then set $A_1 = 1$; otherwise, set $A_2 =1$.  In either case we obtain a solution in terms of radicals to an equation of the form $X^n A X^m = B$ with $m\neq n$, which is impossible by Lemma~\ref{XnAXthm}.

In the case $k=1$, if $n_0 \neq n_1$ then there is no solution by radicals by Lemma~\ref{XnAXthm}, while if $n_0 = n_1=n$, then the word $w=X^n A X^n$ is totally decomposable, so the equation $X^n A X^n = B$ has a solution by radicals by Lemma~\ref{solvablewordprop}.
\end{proof}

\section*{Acknowledgements}
The authors thank George Bergman, Pavel Etingof, Martin Kassabov, Kiran Kedlaya, Igor Klep, Alexei Miasnikov, Bjorn Poonen, and Kate Stange for helpful conversations.  We thank an anonymous referee for contributing Lemma~\ref{refereesobservation} and for a number of other suggestions which improved the paper.

\bibliographystyle{abbrv}
\bibliography{HLR-UniDivGrps}\label{sec:biblio}

\end{document}